\documentclass[letterpaper, 10 pt, conference]{ieeeconf}  

\IEEEoverridecommandlockouts                              

\usepackage{amsmath,graphicx}
  \usepackage{amsthm}
\usepackage{url}
\usepackage{booktabs}

\usepackage{pgfplots}

\usepackage{amssymb}    
\usepackage{mathtools} 

\usepackage[caption=false,font=normalsize,labelfont=sf,textfont=sf]{subfig}

\usepackage[ruled,norelsize,lined,linesnumbered]{algorithm2e}
\usepackage{algorithmic}

\newtheorem{remark}{Remark}
\newtheorem{proposition}{Proposition}

\pgfplotsset{compat=1.18}
\usetikzlibrary{patterns.meta, arrows.meta}
\definecolor{oceanblue}{RGB}{53, 138, 190}
\definecolor{palegreen}{RGB}{179, 216, 178}
\definecolor{paleorange}{RGB}{255, 227, 178}
\definecolor{palepurple}{RGB}{242, 204, 218}

\newcommand{\sav}[1]{\textcolor{red}{[#1]\raise 0.5ex \hbox{\footnotesize{SB}}}}

\usepackage{comment}

\title{\LARGE \bf
Voltage Support Procurement in Transmission Grids:\\Incentive Design via Online Bilevel Games
}

\author{Zhisen Jiang$^{1}$, Saverio Bolognani$^{2}$, Giuseppe Belgioioso$^{3}$
\thanks{$^{1}$ Shenzhen International Graduate School (SIGS), Tsinghua University, Shenzhen 518071, China.
        {\tt\small jzs22@mails.tsinghua.edu.cn}}%
\thanks{$^{2}$ Automatic Control Laboratory, ETH Zurich, 8092 Zurich, Switzerland.
        {\tt\small bsaverio@ethz.ch}}%
\thanks{$^{3}$ Division of Decision and Control Systems (DCS), KTH Royal Institute of Technology, Stockholm, Sweden.
        {\tt\small giubel@kth.se}}%
\thanks{This work was supported by the Swiss Federal Office of Energy via the grant SI/502734 MAESTRO, by the Swiss National Science Foundation under the NCCR Automation (grant agreement 51NF40\textunderscore225155), and by the Wallenberg AI, Autonomous Systems and Software Program (WASP) funded by the Knut and Alice Wallenberg Foundation.}
}

\begin{document}

\maketitle
\thispagestyle{empty}
\pagestyle{empty}

\begin{abstract}
The integration of distributed energy resources into transmission grid operations presents a complex challenge, particularly in the context of reactive power procurement for voltage support.
This paper addresses this challenge by formulating the voltage regulation problem as a Stackelberg game, where the Transmission System Operator (TSO) designs incentives to guide the reactive power responses of Distribution System Operators (DSOs).
We utilize a gradient-based iterative algorithm that updates the incentives to ensure that DSOs adjust their reactive power injections to maintain voltage stability. We incorporate principles from online feedback optimization to enable real-time implementation, utilizing voltage measurements in both TSO's and DSOs' policies.
This approach not only enhances the robustness against model uncertainties and changing operating conditions but also facilitates the co-design of incentives and automation. 
Numerical experiments on a 5-bus transmission grid demonstrate the effectiveness of our approach in achieving voltage regulation while accommodating the strategic interactions of self-interested DSOs.
\end{abstract}

\section{Introduction}

Modern distribution grids host an increasing number of distributed energy resources: micro-generators, batteries, and controllable loads. 
These resources need to be incorporated into transmission grid operation, and most efforts have focused on procuring aggregate services (e.g., primary frequency regulation) from them. 
However, this fine network of controllable resources can provide more complex services, such as controllable reactive power for real-time voltage regulation at the transmission grid level.
These services hold great value in the transition to a grid dominated by renewable energy sources: as large power plants are being replaced, they also become unavailable to regulate grid voltages; transmission grid expansion requires additional reactive power compensation; and local availability of reactive power allows the use of existing tie lines to exchange more valuable active power depending on the availability of clean generation.

As Transmission System Operators (TSOs) do not have direct control over the resources connected to the distribution grids that can be dispatched by the Distribution System Operators (DSOs), various TSO-DSO coordination methods have been proposed for voltage regulation. 
Some approaches have the TSO set voltage stability targets for the DSO to achieve \cite{Wang2023}, while others rely on setting reactive power set points for DSOs' responses \cite{OulisRousis2020, Radi2022}. 
A three-stage process is also used, where the DSO submits reactive power flexibility, the TSO dispatches injections accordingly, and the DSO re-dispatches locally \cite{Alizadeh2023,Usman2023}. 
In some cases, a combination of a prescribed behavior and an economic compensation scheme is proposed. For example, droop-like volt-VAr curves are being proposed in multiple grid codes \cite{ENTSO2012,IEEE1547,DKE2011}, which \emph{de facto} imposes a specific response by the energy resources rather than allowing them to freely respond to an incentive. 
For this reason, the joint response can be inefficient, i.e., the service is not procured from the cheapest sources, and the total capacity may not be used in full.

One notable solution is the voltage support incentive scheme proposed by the Swiss TSO Swissgrid (see Figure~\ref{swissgrid}), where DSOs are remunerated for their reactive power injection depending on whether it alleviates or aggravates the voltage regulation problem \cite{swissgrid2019concept,swissgrid2024tariffs}. 
Individual DSOs respond to the incentives via manual or automated responses (see \cite{ortmann2023deployment} for an example).
This incentive has been in place for several years (previously in a slightly different form, see \cite{Hug2021SwissCase}).

\begin{figure}
    \centering
    \includegraphics[width=\columnwidth]{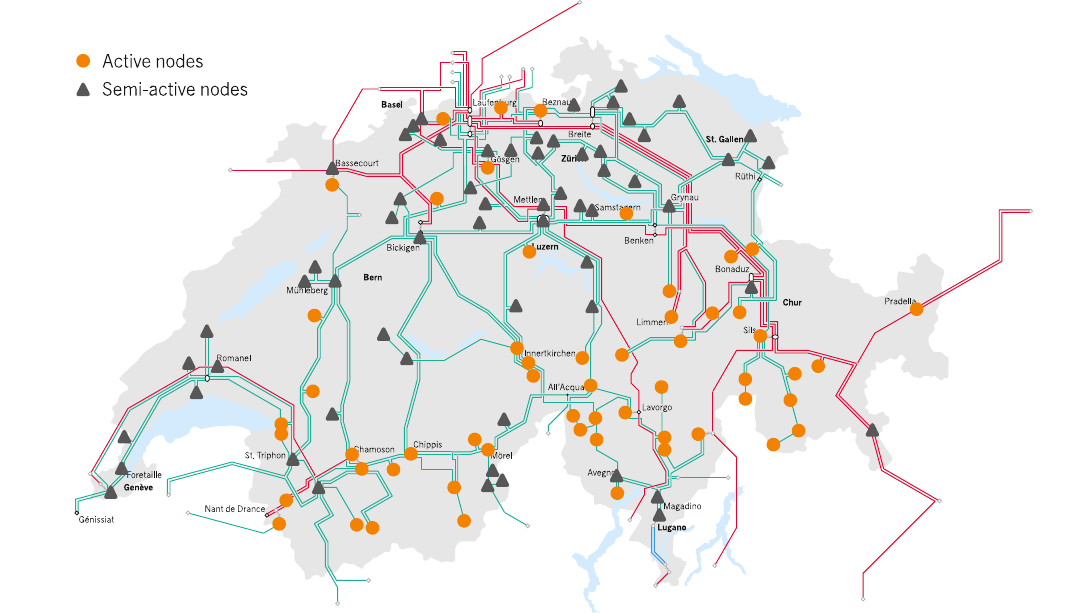}\\
    \caption{The voltage-regulation incentive implemented in the Swiss grid. DSOs are classified as active or semi-active nodes, corresponding to slightly different incentives, 
    and they are remunerated if their reactive power demand contributes to the voltage support goal.}
    \label{swissgrid}
\end{figure}

A crucial aspect of this scheme is that it couples the decision of single participants (how much reactive power to inject) to the state of the grid (the voltage at the substation), which is, in turn, affected by the decisions of all DSOs. Mathematically, this creates a noncooperative game, i.e., a decision problem where self-interested agents are coupled in their rewards. 
Because the TSO decides the parameters of the incentive scheme (similarly to \cite{cavraro2024feedback,klein2024hierarchical}), the entire architecture becomes a Stackelberg game with multiple followers.

In this paper, we first review the motivating case study of the Swiss voltage procurement scheme (Section~\ref{sec:swissgrid}).
Then, in Section~\ref{sec:game_formulation}, we use game theory's formalism to provide a rigorous mathematical formulation of the problem faced by the TSO (in deciding the incentives) and the DSOs (in responding optimally).
In Section~\ref{sec:solution} we present the main result: we consider the case in which incentives can be updated in real time and propose a computational solution that allows the TSO to adjust incentives based on the DSO response so that the desired grid voltage profile can be achieved.
The key tool in this design is adapted from \cite{GrontasDoerfler2023} and consists of a gradient-based iterative algorithm for the optimal intervention in Stackelberg games.
We finally illustrate the performance of this hierarchical service procurement solution in simulations (Section~\ref{sec:case}) and discuss the remaining open challenges in the Conclusions.

\section{A motivating case study: The Swiss voltage support procurement scheme}
\label{sec:swissgrid}

\begin{figure}[tb]
    \centering
    \includegraphics[width=0.8\columnwidth]{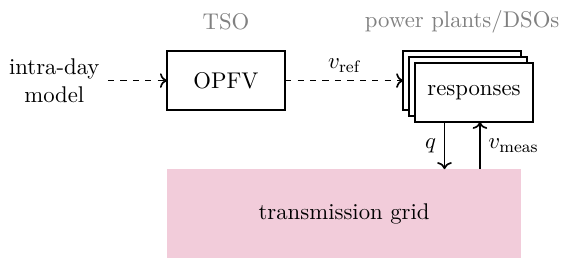}
    \caption{Hierarchical architecture implemented in the Swiss grid: the TSO produces optimal daily voltage reference schedules by solving a specialized OPF program. These voltage references are communicated to the DSO. DSOs respond by deciding their reactive power demand, based  on real-time voltage measurements (as they receive a financial incentive that is a function of local voltage and the reactive power demand).}
    \label{swiss-architecture}
\end{figure}


We briefly review the procedure that is currently used by Swissgrid to procure voltage support services from Swiss DSOs.
As shown in Figure~\ref{swiss-architecture}, it consists of a hierarchical structure.
Based on an intra-day model, which includes forecasts of the power demands over the day, the system operators solved an Optimal Power Flow problem called OPFV in Figure~\ref{swiss-architecture}. 
This optimization program computes a voltage profile for all buses of the transmission grid so that voltage limits and line congestion constraints are satisfied everywhere in the network, considering the model given by the grid power flow equations. 
%
Then, the voltage profiles are communicated the DSOs as voltage reference signals, denoted by $v_\text{ref}$ in Figure~\ref{swiss-architecture}.

\begin{figure}[tb]
    \centering
    \resizebox{!}{35mm}{\begin{tikzpicture}
\begin{axis}[
    width=0.45\textwidth,
    xlabel={$q$},
    ylabel={$v-v_\text{ref}$},
    axis lines=middle,
    domain=-12:12,
	ticks=none,
    enlargelimits=true,
    axis line style={->, thick},
    xlabel style={anchor=north},
    ylabel style={anchor=west},
    axis on top,
    axis line style={-latex}
]

\def\xlim{12}
\def\ylim{3}
\def\deltavtol{1}
\def\deltavfree{1}

\addplot[pattern={Lines[angle=45,line width=2.5pt, distance=8pt]}, pattern color=palegreen, draw=none] coordinates {(-\xlim, -\deltavtol) (\xlim, -\deltavtol) (\xlim, \deltavtol) (-\xlim, \deltavtol)};
\addplot[pattern={Lines[angle=45,line width=2.5pt, distance=8pt]}, pattern color=paleorange, draw=none] coordinates {(0, -\deltavtol-\deltavfree) (\xlim, -\deltavtol-\deltavfree) (\xlim, -\deltavtol) (0, -\deltavtol)};
\addplot[pattern={Lines[angle=45,line width=2.5pt, distance=8pt]}, pattern color=paleorange, draw=none] coordinates {(-\xlim, \deltavtol) (0, \deltavtol) (0, \deltavtol+\deltavfree) (-\xlim, \deltavtol+\deltavfree)};
\addplot[fill=palegreen, draw=none] coordinates {(0, 0) (\xlim, 0) (\xlim, \ylim) (0, \ylim)};
\addplot[fill=palegreen, draw=none] coordinates {(-\xlim, -\ylim) (0, -\ylim) (0, 0) (-\xlim, 0)};

\end{axis}
\draw[<->, thick] (2.3, 3.4) -- (2.3, 4.1) node[midway, left] {$\Delta v_\text{free}$};
\draw[<->, thick] (2.3, 2.65) -- (2.3, 3.35) node[midway, left] {$\Delta v_\text{conform}$};
\end{tikzpicture}}
    \resizebox{!}{35mm}{\begin{tikzpicture}
\begin{axis}[
    width=0.45\textwidth,
    xlabel={$q$},
    ylabel={$v-v_\text{ref}$},
    axis lines=middle,
    domain=-12:12,
	ticks=none,
    enlargelimits=true,
    axis line style={->, thick},
    xlabel style={anchor=north},
    ylabel style={anchor=west},
    axis on top,
    axis line style={-latex}
]

\def\xlim{12}
\def\ylim{3}
\def\deltaqfree{2.5}
\def\deltavfree{1}

\addplot[pattern={Lines[angle=45,line width=2.5pt, distance=8pt]}, pattern color=paleorange, draw=none] coordinates {(-\deltaqfree,\ylim) (\deltaqfree, \ylim) (\deltaqfree, -\ylim) (-\deltaqfree, -\ylim)};
\addplot[pattern={Lines[angle=45,line width=2.5pt, distance=8pt]}, pattern color=paleorange, draw=none] coordinates {(-\xlim, \deltavfree) (\xlim, \deltavfree) (\xlim, -\deltavfree) (-\xlim, -\deltavfree)};
\addplot[fill=palegreen, draw=none] coordinates {(\deltaqfree, \deltavfree) (\xlim, \deltavfree) (\xlim, \ylim) (\deltaqfree, \ylim)};
\addplot[fill=palegreen, draw=none] coordinates {(-\xlim, -\ylim) (-\deltaqfree, -\ylim) (-\deltaqfree, -\deltavfree) (-\xlim, -\deltavfree)};

\end{axis}
\draw[<->, thick] (3.2, 4.5) -- (2.67, 4.5) node[left] {$\Delta q_\text{free}$};
\draw[<->, thick] (2.3, 2.65) -- (2.3, 3.35) node[midway, left] {$\Delta v_\text{free}$};
\end{tikzpicture}}%
    \caption{Two SwissGrid incentive schemes available to DSOs: active (left) and semi-active (right). Reactive power demands that reduce voltage deviation (green regions) are rewarded, whereas those that increase deviation (white regions) incur penalties. Small tolerance bands accommodate measurement errors.}
    \label{fig:swiss-incentives}
\end{figure}
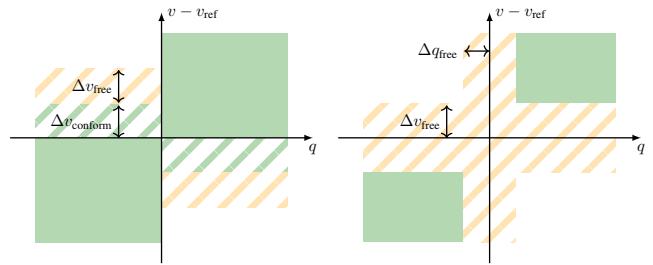
Each DSO is incentivized to track the voltage reference signal according to one of two possible incentive schemes, depending on the specific agreement negotiated. These two incentive schemes are illustrated schematically in Figure~\ref{fig:swiss-incentives}.
Fundamentally, the DSO positive reactive power demand is rewarded if the measured bus voltage $v$ is larger than the reference $v_\text{ref}$, and negative reactive power demand is rewarded if $v<v_\text{ref}$ (green areas in Figure~\ref{fig:swiss-incentives}). Otherwise (white areas), DSOs are charged a financial penalty. 
Small tolerance areas ($\Delta v_\text{free}$, $\Delta v_\text{conform}$, $\Delta q_\text{free}$) extend the region where the DSOs are still rewarded (shaded green areas) and introduce no-reward-no-penalty areas (shaded yellow areas), mainly to account for measurement errors.

Intuitively, it seems that such an incentive scheme promotes good tracking performance by the DSOs, as reactive power is rewarded/penalized depending on whether it alleviates/aggravates the optimal voltage tracking error. 

However, this is not entirely correct. 
Figure~\ref{individual-responses} shows the payment received by two exemplary DSOs in the Swiss grid (red line) as a function of their reactive power demand $q$. 
Crucially, changes in reactive power demand $q$ affect the bus voltage (blue line), and therefore the operating point moves in the reward/penalty plane.
Two issues are evident. First, the reward curve is not concave, leading to a possibly difficult optimization problem to be solved by the DSO. Moreover, the maximal reward is not necessarily achieved when the DSO tracks the voltage reference perfectly. 

\begin{figure}[tb]
    \centering
    \includegraphics[width=\columnwidth]{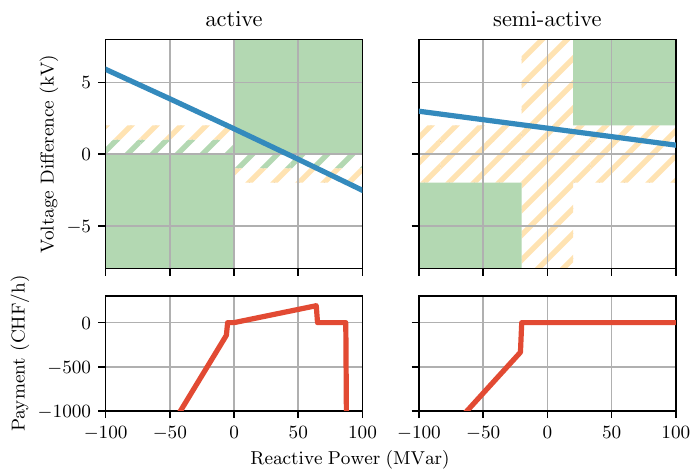}
    \caption{Payment received by two DSOs as a function of their reactive power demand, when the other DSOs maintain their injection. At the Nash Equilibrium, each agent must maximize their reward, which appear to be a discontinuous and non-convex curve.}
    \label{individual-responses}
\end{figure}

%

The analysis of historical data provided by Swissgrid (Figure~\ref{responses}) indicates that reactive power responses from DSOs do not consistently result in positive financial rewards and often lead to significant voltage-tracking errors. An exception is represented by traditional power plants, which typically achieve near-perfect voltage tracking, except when limited by their reactive power constraints. It also appears from the data that many DSOs may frequently reach their reactive power limits, resulting in penalties. However, these constraints are currently not incorporated into the OPFV program, determining reference profiles $v_{\text{ref}}$ that are infeasible to track.

\begin{figure}
    \includegraphics[width=\columnwidth]{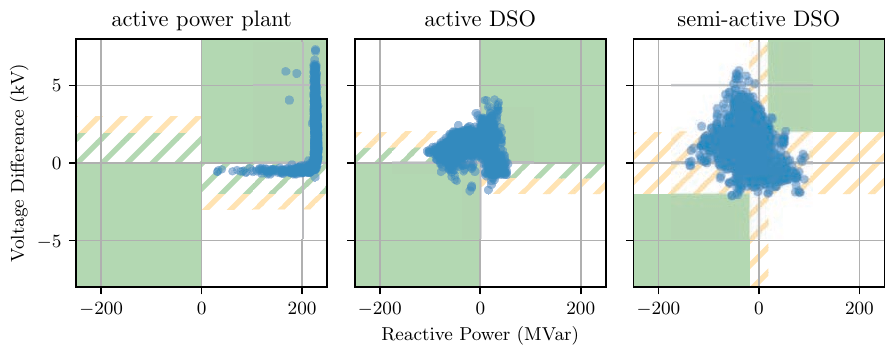}
    \caption{The voltage-regulation incentive implemented in the Swiss grid. DSOs are classified as active or semi-active nodes, which correspond to slightly different incentives. Power plants have a different (and stricter) incentive scheme. As represented in the plots, nodes are remunerated if their reactive power demand contributes to the voltage support goal. Data show that the compliance of nodes is not uniform across the participants.}
    \label{responses}
\end{figure}

This sequential computation of optimal voltage references via the OPFV program and the \emph{ex-post} reward of the tracking performance has clear limitations. In particular, it neglects valuable real-time information about the DSOs' constraints and incentive responses when setting voltage references. These drawbacks  motivate a more effective principled co-design of incentives and DSO responses.


\section{Game-theoretic formulation} \label{sec:game_formulation}

In our mathematical abstraction of the reactive power procurement problem, we adopt the same hierarchical architecture as depicted in Figure~\ref{swiss-architecture}, but we design a different incentive.
Crucially, we introduce an additional degree of flexibility by allowing the TSO to iteratively update the voltage reference signal $v_{\text{ref}}$ in real-time, in response to observed DSO responses, i.e., their reactive power. This approach enables us to account for both the DSOs' reward sensitivity and reactive power constraints, resulting in a more efficient and responsive voltage support mechanism.

Let us consider a transmission network with $n$ buses (each one corresponding to a different DSO for simplicity) and define these vectors of $\mathbb R^n$:
\begin{itemize}
    \item[$q$] the vector of bus reactive power demands;
    \item[$p$] the vector of bus active power demands;
    \item[$v$] the vector of bus voltage magnitudes;
    \item[$v_\text{ref}$] the vector of bus voltage references.    
\end{itemize}
We use the subscript $i$ to refer to the quantity corresponding to bus $i$, and $-i$ to refer to all buses except $i$.

The overall decision problem jointly faced by the TSO and the DSOs consists in the following single-leader multi-follower Stackelberg game:
\begin{subequations}\label{eq:stackelberg_game}
\begin{align}
\min_{v_\text{ref}, q} \;\; & \sum_i \Pi(q_i,v_i(q_i,q_{-i},p),v_\text{ref}) 
    \tag{1a}\label{eq:outer_obj}\\[0.5em]
\text{s.t.}\;\; &  \underline{v}\leq v(q,p) \leq \overline{v} 
    \tag{1b}\label{eq:v_constr}\\[0.5em]
& \forall i\in \mathcal{I}: \notag \\
& \quad  q_i = \underset{\; \; \xi_i}{\text{argmin}}\;\;  c_i(\xi_i)-\Pi(\xi_i,v_i(\xi_i,q_{-i},p),v_\text{ref}) 
    \tag{1c}\label{eq:inner_obj}\\[-0.2em]
& \hspace*{1.7cm} \text{s.t.}\; \; \underline{q_i}\leq \xi_i \leq \overline{q_i} 
    \tag{1d}\label{eq:q_constr}
\end{align}
\end{subequations}
where
$\Pi(q_i, v_i, v_{\text{ref}})$ is the economic reward parametrized in $v_{\text{ref}}$, the decision variable of the TSO; $v(q,p)$ models the voltage as a function of the power demands; and $c_i(\xi_i)$ is the cost of producing $\xi_i$ unit of reactive power.

Notice that, for clarity, we use $\xi_i$ when we refer to reactive power demand at bus $i$ as a decision variable (for example, inside an algorithm), while $q_i$ is the solution of the inner problem \eqref{eq:inner_obj}-\eqref{eq:q_constr}.

We consider the incentive function
\begin{equation}	\Pi(q_i,v_i,v_i^\text{ref})= \gamma (v_i - v_i^\text{ref}) q_i
\end{equation}
where $\gamma$ represents a tariff coefficient set by the TSO, and therefore $\gamma (v_i - v_i^\text{ref})$ determines the price per unit of reactive power based on the voltage deviation. 
Therefore, it incentivizes DSOs to provide reactive power support when the local voltage is below the reference and encourages reactive power absorption when the voltage is above the reference.
Compared to the Swissgrid incentive reviewed in Section~\ref{sec:swissgrid}, the incentive is proportional to the voltage tracking error and not simply dependent on its sign.

For the analysis, we consider a linearized version of the power flow equations
\begin{equation}
\label{eq:linPF}
v=Rp+Xq +v_0
\end{equation}
where $R,X$ are derived from the resistance and reactance matrices, and $v_0$ is the baseline voltage profile of the system in the absence of any power consumptions.

\subsection{Properties of the lower level game (DSOs)}



At the lower level of \eqref{eq:stackelberg_game}, the DSOs aim at maximizing their profit from participating in the voltage support scheme, while adhering to their reactive power limits.
This is formalized by the constrained optimization problem \eqref{eq:inner_obj}--\eqref{eq:q_constr}.

The key factor that determines a coupling in the decisions of the DSOs lies in the incentive function $\Pi(q, v_i, v_i^\text{ref})$, which uses the real-time voltage error $(v_i - v_i^\text{ref})$ to change the remuneration for procurement. 
Since $v_i\left( q,p \right)$ is determined by all reactive power demands $q = (q_i, q_{-i})$, each DSO remuneration depends on the reactive power responses of the others $q_{-i}$. 
Collectively, these DSOs optimization problems constitute a noncooperative game parametrized in $v_{\text{ref}}$.

In practice, the cost of producing/consuming reactive power depends on the DSO's energy assets and contractual agreements with the independent stakeholders in their distribution network.
Here, we choose as a quadratic function of the form $c_i(\xi _i)=\frac{1}{2}C_i\xi _{i}^{2}$. 

In the remainder of this section, we analyze the properties of the lower-level game and establish conditions for the existence and uniqueness of its Nash equilibrium.

\begin{proposition} \label{prop:convexity}
    For each DSO $i$, the cost function $f_i\left( v_{i}^{\mathrm{ref}},\xi_i,\xi_{-i} \right) \coloneqq c_i\left( \xi _i \right) -\gamma \left( v_i\left( \xi ,p \right) -v_{i}^{\mathrm{ref}} \right) \xi _i$ is convex with respect to $\xi_i$ for any fixed $\xi_{-i}$ and $-v_{i}^{\mathrm{ref}}$.
\end{proposition}
\begin{proof}
    With linearized grid model \eqref{eq:linPF}, the voltage at node $i$ is given by:
$$
v_i = X_{ii} \xi_i + \sum_{j \neq i} X_{ij} \xi_j + v_{0,i} + R_{i:} p,
$$
where $X_{ii}$ represents the self-sensitivity of voltage at bus $i$ to its own reactive power consumption, and $X_{ij}$ is the mutual-sensitivity.
Substituting $v_i$ into the objective function, we have:
\begin{equation}
    \label{eq:f_i}
    f_i( v_{i}^{\text{ref}},\xi_i,\xi_{-i}) =\frac{1}{2}\left( C_i-2\gamma X_{ii} \right) \xi _{i}^{2}-\gamma K\xi _i,
\end{equation}
where $K = \sum_{j \neq i} X_{ij} \xi_j + v_{0,i} + R_{i:} p - v_i^{\text{ref}}$ is a constant w.r.t. $\xi_i$.
As $C_i>0$,  $\gamma>0$, and diagonal elements of $X$ are negative (which is typical in power networks), we have:
\begin{equation*}
    \frac{d^2 f_i}{d \xi_i^2} = C_i - 2 \gamma X_{ii} > 0,
\end{equation*}
which proves $f_i\left( v_{i}^{\text{ref}},\xi_i,\xi_{-i} \right)$ is convex.
\end{proof}

Before discussing the existence and uniqueness of the lower-level equilibrium, we define the pseudo-gradient mapping $F(v_{\text{ref}}, \cdot)$ as follows:
\begin{equation*}
    \label{eq:pseudogradient}
    F(v_{\text{ref}}, \cdot) = \left( \nabla_{\xi_i}f_i(v_{i}^{\text{ref}},\cdot) \right)_{i\in\mathcal{N}}.
\end{equation*}
The following result establishes sufficient conditions for the existence and uniqueness of the Nash equilibrium.

\begin{proposition} \label{prop:existence&uniqueness}
    Assume $\lambda_{\min}(C - \gamma \tilde{X}) \geq \mu > 0$. Then, the pseudo-gradient mapping $F(v_{\mathrm{ref}}, \cdot)$ is $\mu$-strongly monotone and $L_F$-Lipschitz continuous, with $L_F =\|C - \gamma \tilde{X}\|$.
\end{proposition}

\begin{proof}
    From \eqref{eq:f_i}, the pseudo-gradient mapping reads as
    \begin{equation*}
        F(v_{\text{ref}}, \xi) = (C - \gamma \tilde{X}) \xi - \gamma (v_0 + R p - v_{\text{ref}}),
    \end{equation*}
    where $C$ is the diagonal matrix of cost coefficients and $\tilde{X} \coloneqq X + \operatorname{diag}(X_{ii})_{i \in \mathcal{N}}$.
    
    To analyze its properties, we compute the Jacobian matrix:
    \begin{equation*}
        J_F \coloneqq \nabla_{\xi} F(v_{\text{ref}}, \xi) = C - \gamma \tilde{X}.
    \end{equation*}
    Because both $C$ and $\tilde{X}$ are symmetric, $J_F$ is a symmetric matrix.
    
    Since $\lambda_{\min}(C - \gamma \tilde{X}) \geq \mu > 0$, the Jacobian $J_F$ is positive definite. Consequently, for any $\xi_1, \xi_2$, we have:
    \begin{equation*}
        (\xi_1 - \xi_2)^\top (F(v_{\text{ref}}, \xi_1) - F(v_{\text{ref}}, \xi_2)) \geq \mu \|\xi_1 - \xi_2\|^2,
    \end{equation*}
    which establishes that $F(v_{\text{ref}}, \cdot)$ is $\mu$-strongly monotone.

    Furthermore, from the definition of $F$, it follows that:
    \begin{equation*}
        \|F(v_{\text{ref}}, \xi_1) - F(v_{\text{ref}}, \xi_2)\| \leq \|C - \gamma \tilde{X}\| \|\xi_1 - \xi_2\|, \quad \forall \xi_1, \xi_2.
    \end{equation*}
    Denoting the upper bound of the constant matrix norm by $L_F \coloneqq \|C - \gamma \tilde{X}\|$, we conclude that $F(v_{\text{ref}}, \cdot)$ is $L_F$-Lipschitz continuous.
\end{proof}

\begin{remark}
The condition $\lambda_{\min}(C - \gamma \tilde{X}) \ge \mu > 0$ is automatically verified if $\tilde{X}$ is negative definite, which is the case of radial networks \cite[Lemma 1]{Farivar2013}.
For meshed network where $\tilde{X}$ is not negative definite, the condition can be enforced in the design of the incentive by selecting $\gamma <\frac{c_{\min}}{\lambda _{\max}\left( \tilde{X} \right)}$.
\end{remark}

Since \eqref{eq:q_constr} is affine, nonempty and satisfies Slater's condition, and the pseudo-gradient is strongly monotone and Lipshitz continuous by Proposition \ref{prop:existence&uniqueness}, we can invoke \cite[Theorem\ 2.3.3(b)]{facchinei2003finite} to show existence and uniqueness of the lower-level Nash equilibrium $q^*(v_{\text{ref}})$, for any choice of $v_{\text{ref}}$.

\section{Co-design of incentives and DSOs' response} \label{sec:solution}
%
%
%

Based on this game-theoretic formulation, we propose an online protocol to update $v_{\text{ref}}$ as a dynamic incentive signal, and simultaneously automate the response of the DSOs, so that they jointly converge to the optimal grid operation defined by the Stackelberg game in \eqref{eq:stackelberg_game}.

Our proposed protocol is derived by deploying and tailoring the BIG Hype algorithm \cite{GrontasDoerfler2023} as an online control strategy to solve \eqref{eq:stackelberg_game}. This online protocol simultaneously computes the lower-level Nash equilibrium and its Jacobian, namely, the sensitivity of the Nash equilibrium with respect to the incentive parameters. The latter is crucial for calculating the gradient of the upper-level objective. Finally, this gradient then enables the TSO to update the incentive signal $v_{\text{ref}}$ to optimize the objective function  \eqref{eq:outer_obj} in real time. Importantly, our scheme retains the original hierarchical and distributed structure of the problem, ensuring scalability and preserving privacy.

\subsection{Inner Loop: Equilibrium and Sensitivity Estimation}

At the lower level, we estimate the Nash equilibrium of the DSOs' game and its sensitivity with respect to the TSO's incentive parameter $v_{\text{ref}}$. The estimation proceeds as follows.


The Nash equilibrium $q^*(v_{\text{ref}})$ of the lower level game \eqref{eq:inner_obj}--\eqref{eq:q_constr} is estimated via the following fixed-point iteration:
\begin{equation}
    \left(\forall i\in \mathcal{N} \right) \quad \tilde{\xi}^{l+1}_{i}=h_i\left( v_{i}^{\text{ref}}, \tilde{\xi}^l \right), \label{eq:xi_update}
\end{equation}
where  $h_i$ is a projected pseudo-gradient update of the form
\begin{equation*}
h_i\left( v_{i}^{\text{ref}}, \xi \right) \coloneqq \mathbb{P}_{\left[ \underline{q}_i, \overline{q}_i \right]} \left( \xi_i - \eta F_i\left( v_{i}^{\text{ref}}, \xi \right) \right),
\end{equation*}
and $F_i$ is the pseudo-gradient:
\begin{align}
    F_i\left( v_{i}^{\text{ref}},\xi \right) &\coloneqq \nabla_{\xi_i}f_i\left( v_{i}^{\text{ref}},\xi \right) \notag \\
    &= \nabla c_i\left( \xi_i \right) - \gamma \left( v_i - v_{i}^{\text{ref}} \right) - \gamma \xi_i \nabla_{\xi_i} v_i. \label{eq:Fi}
\end{align}
The term $\nabla_{\xi_i} v_i$ represents the sensitivity of the voltage magnitude at DSO $i$'s connection point with respect to its reactive power consumption. 

Simultaneously, we estimate the equilibrium sensitivity with respect to the incentive parameter $v_{i}^{\text{ref}}$ using:
\begin{equation}
  \tilde{s}_{i}^{l+1} = \mathbf{J}_2 h_i\left( v_{i}^{\text{ref}},\tilde{\xi}^{l+1} \right) \tilde{s}_{i}^{l} + \mathbf{J}_1 h_i\left( v_{i}^{\text{ref}},\tilde{\xi}^{l+1} \right). \label{eq:s_update}
\end{equation}

Let $g_i\left( \cdot \right) \coloneqq \mathbb{P}_{\left[ \underline{q}_i,\overline{q}_i \right]} \left[ \cdot \right]$ denote the projection operator. 
Then, the partial Jacobians in \eqref{eq:s_update} can be expressed as
\begin{align}
    \mathbf{J}_1 h_i &= -\eta \nabla g_i\left( \xi_i - \eta F_i\left( v_{i}^{\text{ref}},\xi \right) \right) \mathbf{J}_1 F_i\left( v_{i}^{\text{ref}},\xi \right), \label{eq:J1h}
    \\
    \mathbf{J}_2 h_i &= \nabla g_i\left( \xi_i - \eta F_i\left( v_{i}^{\text{ref}},\xi \right) \right) \left( I - \eta \mathbf{J}_2 F_i\left( v_{i}^{\text{ref}},\xi \right) \right), \label{eq:J2h}
\end{align}
where the operator $\nabla g_i$ is defined as\footnote{This is an abuse of notation as $g$ is not everywhere differentiable.}
\begin{align}
    \nabla g_i(x)=\begin{cases}
	1, \quad \underline{q}_i\le x \le \overline{q}_i\\
	0, \quad \mathrm{otherwise}.\\
\end{cases}
\end{align}

Asymptotically, we can show that the iterations \eqref{eq:xi_update} and \eqref{eq:s_update} converge to $q^{*}(v_{\text{ref}})$ and its sensitivity $\mathbf{J}q^*(v_{\text{ref}})$.

Note that the iteration \eqref{eq:xi_update} can be interpreted as a decomposition of the joint fixed-point iteration $\tilde{\xi}^{l+1} = h(v_{\text{ref}}, \tilde{\xi}^l)$, where the mapping $h(v_{\text{ref}}, \xi) \coloneqq (h_1(v_{\text{ref}}, \xi), \dots, h_N(v_{\text{ref}}, \xi))$ jointly updates all DSOs' reactive power decisions. 
Moreover, the mapping $h(v_{\text{ref}}, \cdot)$ is contractive with contraction constant $\theta \coloneqq \sqrt{1 - \eta (2\mu - \eta L_F^2)} < 1$, for any step size choice $\eta < 2\mu/L_F^2$. 
Consequently, the iteration \eqref{eq:xi_update} converges linearly to $q^*(v_{\text{ref}})$ with rate $\theta$.
The convergence of the sensitivity $\mathbf{J}q^*(v_{\text{ref}})$ follows from \cite[Prop.~2]{GrontasDoerfler2023}.

\subsection{Outer Loop: Incentive Update by the TSO}

The TSO uses the estimated reactive power Nash equilibrium and its sensitivity to update the incentive signal $v_{\text{ref}}$.

To simplify the TSO update, we relax the voltage constraint \eqref{eq:v_constr} using the penalty function  
\begin{equation}
    \phi(v) = \rho \max \left\{ 0,v-\overline{v} \right\} ^2 + \rho \max \left\{ 0, \underline{v}-v \right\} ^2,
\end{equation}
where $\rho>0$ is a tuning parameter.

We then define the augmented objective function as  
\begin{equation}
\varphi \left( v_{\text{ref}}, v, q \right) \coloneqq \sum_i \Pi(q_i, v_i, v_{\text{ref}}) + \phi(v).
\end{equation}
The Stackelberg game in \eqref{eq:stackelberg_game} can be compactly cast as  
\begin{equation}
    \label{eq:single_level}
    \min_{v_\text{ref}} \quad  \varphi \left( v_{\text{ref}}, v^*, q^* \right) =: \varphi_e\left( v_{\text{ref}} \right),
\end{equation}
where $q^* = q^*(v_{\text{ref}})$ and $v^* = v(q^*(v_{\text{ref}}),p)$ for brevity.

Whenever $\varphi_e(\cdot)$ is differentiable at $v_{\text{ref}}$, we can exploit the chain-rule to compute the gradient of $\nabla \varphi_e\left( v_{\text{ref}} \right)$, commonly referred to as hyper-gradient, yielding 
\begin{align}
\label{eq:HG}
    \nabla \varphi_e\left( v_{\text{ref}} \right) = &
    \nabla_1 \varphi \left( v_{\text{ref}}, v^*, q^* \right) \notag \\
    &+ \left( \mathbf{J}_q v \cdot s^* \right)^\top \nabla_2 \varphi \left( v_{\text{ref}}, v^*, q^* \right) \notag \\
    & + ({s^*})^\top \nabla_3 \varphi \left( v_{\text{ref}}, v^*, q^* \right),
\end{align}
where $s^* = \mathbf{J}q^*(v_{\text{ref}})$ is the sensitivity of the Nash equilibrium to $v_{\text{ref}}$, and
$\mathbf{J}_q v$ is the sensitivity of voltage magnitudes to reactive power consumption.\footnote{Typically, this is approximated by the grid's reactance matrix in linearized power flow models.} A similar expression holds when $\varphi_e(\cdot)$ is not differentiable at $v_{\text{ref}}$, where standard Jacobians are replaced with elements of the conservative Jacobian \cite{bolte2021nonsmooth}. For the sake of simplicity, we will not discuss this case here and instead refer the interested reader to the proof of
Theorem 2 in \cite{GrontasDoerfler2023} for a detailed technical analysis.

In practice, to update $v_{\text{ref}}$, we use an approximate version of the hypergradient \eqref{eq:HG}, where the Nash equilibrium $q^*$ and its sensitivity $s^*$ at are substituted with their online estimates, \eqref{eq:xi_update} and \eqref{eq:s_update}, respectively,  yielding 
\begin{align}
\hat{\nabla} \varphi_e \coloneqq &
\nabla_1 \varphi \left( v_{\text{ref}}^{k}, v^k, \xi^k \right) +
\left( \mathbf{J}_q v \cdot s^k \right)^{\mathrm{T}} \nabla_2 \varphi \left( v_{\text{ref}}^{k}, v^k, \xi^k \right) \notag \\
& + \left( s^k \right)^{\mathrm{T}} \nabla_3 \varphi \left( v_{\text{ref}}^{k}, v^k, \xi^k \right), \label{eq:hypergrad approx}
\end{align}
with the component gradients explicitly given by
\begin{align}
    \nabla_1 \varphi \left( v_{\text{ref}}, v^k, \xi^k \right) &\coloneqq \left[ -\gamma \xi_{i}^{k} \right]_{i\in \mathcal{N}},
    \\
    \nabla_2 \varphi \left( v_{\text{ref}}, v^k, \xi^k \right) &\coloneqq \left[ \gamma \xi_{i}^{k} + \nabla_{v^{k}_{i}}\phi(v^{k})\right]_{i\in \mathcal{N}},
    \\
    \nabla_3 \varphi \left( v_{\text{ref}}, v^k, \xi^k \right) &\coloneqq \left[ \gamma \left( v_{i}^{k} - v_{i}^{\text{ref},k} \right) \right]_{i\in \mathcal{N}}. \label{eq:nabla_phi_3}
\end{align}
Here, $\xi^k$ and $s^k$ are obtained from the inner loop's equilibrium and sensitivity estimation updates.

Finally, the TSO incentive signal is updated as
\begin{equation}
    v_{\text{ref}}^{k+1} = v_{\text{ref}}^{k} - \epsilon^k \hat{\nabla} \varphi_e \left( v_{\text{ref}}^{k} \right), \label{eq:incentive_update}
\end{equation}
where $\left\{ \epsilon ^k \right\} _{k\in \mathbb{N}}$ is a sequence of step sizes.
\subsection{Online Implementation for Incentive Automation}
While theoretically sound, this approach faces practical implementation challenges due to modeling uncertainties and the need for accurate real-time knowledge of the system state. 

We can significantly enhance implementation robustness by adopting principles from online feedback optimization \cite{ortmann2023deployment, HauswirthDoerfler21c, tutorialMultiAreaOFO, BelgioiosoDoerfler22}.
The key insight is that both the inner loop's gradient \eqref{eq:Fi} and outer loop's gradient \eqref{eq:nabla_phi_3} depend on the real-time voltage error $(v_i - v_i^\text{ref})$. 
Rather than relying on a mathematical model to predict these voltage values, we can directly measure the voltage magnitudes $v$ from the physical system. This approach effectively outsources the evaluation of the voltage-reactive power mapping to the physical grid itself, reducing reliance on model information. 
Moreover, real-time measurements automatically incorporate all system dynamics and external disturbances, enhancing the controller's ability to adapt to changing grid conditions. 

\begin{figure}[tb]
	\centering
    \includegraphics[width=1\linewidth]{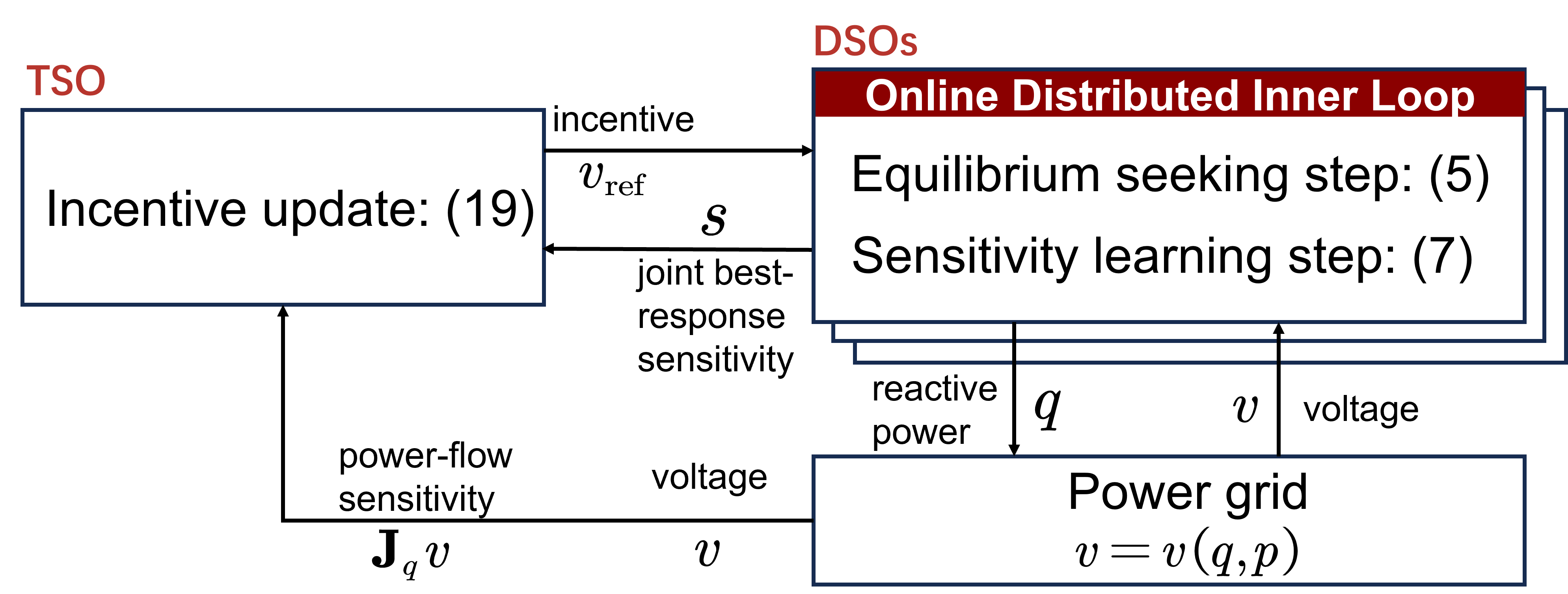}
    \caption{Block diagram of the proposed online incentive automation.}
    \label{fig:block_diagram}
\end{figure}

Figure~\ref{fig:block_diagram} illustrates the resulting architecture.
In the inner loop implementation, DSOs measure real-time voltage $v_i$ at their connection points after each reactive power adjustment. 
These measurements provide direct feedback on how their control actions have affected the grid state, which they then incorporate into their equilibrium-seeking and sensitivity learning steps. Algorithm \ref{alg:inner_loop} presents this distributed measurement-based process in detail.

\begin{algorithm}%
\caption{Online Distributed Inner Loop.}
\label{alg:inner_loop}%
    \textbf{Parameters:} step size $\eta$, tolerance $\sigma$. \\
    \textbf{Input:} $\xi$, $s$, $v_{\text{ref}}$, $\sigma$. \\
    \textbf{Initialization:} $l\leftarrow 0$, termin = false, $\tilde{s}^{l}=s$, $\tilde{\xi}^{l}=\xi$. \\
    \textbf{Iterate until termination} \\[.5em]
    $
                        \begin{array}{l}
                            \left\lfloor
                            \begin{array}{l}
                                \text{For all DSOs }  i \in \mathcal{N} \text{ (in parallel)}: \\ 
    \left\lfloor
    \begin{array}{l}
    \text{  Real-time voltage measurement: } \, v_{i}.\\
    \text{  Equilibrium seeking step:} \, \eqref{eq:xi_update}.\\
    \text{  Jacobian update:} \, \eqref{eq:J1h}, \eqref{eq:J2h}.\\
    \text{  Sensitivity learning step:} \, \eqref{eq:s_update}.\\
    \end{array}
    \right. \\[.5em]
                                %
   \textrm{termin} = \text{\small$\max \left\{ \left\| \tilde{\xi}^{l+1}-\tilde{\xi}^l \right\| ,\left\| \tilde{s}^{l+1}-\tilde{s}^l \right\| \right\} \le \sigma $} \\
                                l \leftarrow l + 1  \\
                            \end{array}
                            \right. 
                        \end{array}	
    $\\[.2em]
    \textbf{Output}: $\xi^{k+1}=\tilde{\xi}^{l}$, $s^{k+1}=\tilde{s}^{l}$ .
\end{algorithm}

Similarly, in the outer loop, the TSO utilizes real-time voltage measurements to inform incentive signal updates. 
By directly measuring grid voltages and using these measurements to approximate the hypergradient, the TSO can adjust the incentive signal in response to actual grid conditions rather than relying solely on model predictions. 
This creates an adaptive feedback mechanism where incentives continuously evolve based on the grid's actual response.
%
The complete co-design of incentives and automation is presented in Algorithm \ref{alg:algorithm}. 

\begin{algorithm}
\caption{Incentives and Automation Co-design.}
\label{alg:algorithm}
    \textbf{Parameters}: step size $\left\{ \epsilon ^k \right\} _{k\in \mathbb{N}}$, tolerance $\left\{ \sigma ^k \right\} _{k\in \mathbb{N}}$.\\
    \textbf{Initialization}: $k\leftarrow 0$, $v_{\text{ref}}^k \leftarrow v_{\text{ref}}^{\text{ini}}$. \\
        \textbf{Iteratively update incentives:}\\
$
            \left \lfloor
            \begin{array}{l}
                \text{DSOs' Estimation of Equilibrium and Sensitivity:}\\[.2em]
                \left|
                            \begin{array}{l}
                             (\xi^{k+1}, \ s^{k+1}) =	\\
                             ~~~~
                \text{\textbf{Online Distributed Inner Loop}} \\[.1em]
                ~~~~			    \left\lfloor
                        \begin{array}{l}
                             \text{Input: } \xi^k,\, s^k,\, v_{\text{ref}}^k,\,\sigma^k\\
                             \text{Output: } \xi^{k+1},\, s^{k+1}\\
                        \end{array}
                            \right. \end{array}
                            \right. \vspace{.25em}
                            \\
                \text{TSO's hypergradient step:} \\[.1em]
                \left|
                \begin{array}{l}
                \text{Real-time voltage measurement: } v^{k}.\\
                \text{Hypergradient approximation: } \eqref{eq:hypergrad approx}. \\
                \text{Incentive update: } \eqref{eq:incentive_update}.\\
                \end{array}
                \right.\\[0.1em]
            k \leftarrow k+1\\
            \end{array}
            \right.
$
\end{algorithm}

\subsection{Algorithm convergence}

In Section \ref{sec:game_formulation}, we established that the lower-level game has a linear-quadratic structure by relying on the linearized grid model \eqref{eq:linPF}.
Under that assumption, we can invoke \cite[Theorem 2]{GrontasDoerfler2023} to prove convergence of the sequence of incentives $\{\varphi_e(v_{\text{ref}}^k)\}_{k\in \mathbb{N}}$ generated by Algorithm \ref{alg:algorithm} to a critical point of the relaxed Stackelberg game in \eqref{eq:single_level}, under appropriate choices of the step size $\{\epsilon^k\}_{k\in \mathbb{N}}$ and tolerance $\{\sigma^k\}_{k\in \mathbb{N}}$ sequences.\footnote{We refer the interested reader to Lemma 6 in \cite{GrontasDoerfler2023} for how to design these sequences.}

However, in the actual implementation, we use real-time voltage measurements directly from the physical grid. This means that, in practice, the algorithm operates with the nonlinear power flow equations rather than their linearized approximation. The convergence guarantees are invalidated if the difference from the linearized model is substantial.
Incorporating the full nonlinear AC power flow equations presents a more realistic yet challenging extension that requires further investigation.

Finally, we note that a constant step size $\epsilon$ would be required in \eqref{eq:incentive_update} to ensure continuous incentive updates in online settings, where the parameters of the bilevel game \eqref{eq:stackelberg_game} (e.g., power injections $p$) vary over time. 
The development of a formal stability and tracking analysis for this case is left for future work.

\section{Numerical experiments}
\label{sec:case}

To validate our proposed approach, we conducted numerical experiments on an illustrative transmission grid with multiple DSOs responding to TSO incentives. 

\subsection{Simulation setup}

The simulations are carried out on a 5-bus transmission network case adapted from \cite{Li2010case5}, which is shown as Figure~\ref{fig:5bus}
The network operates at a single voltage level of 230 kV and includes four DSOs connected at four buses (excluding the slack bus). 
We simulate the grid using Pandapower \cite{Pandapower} to compute the nonlinear power flow solution $v(q,p)$.
\begin{figure}[tbp]
	\centering
    \includegraphics[width=0.9\linewidth]{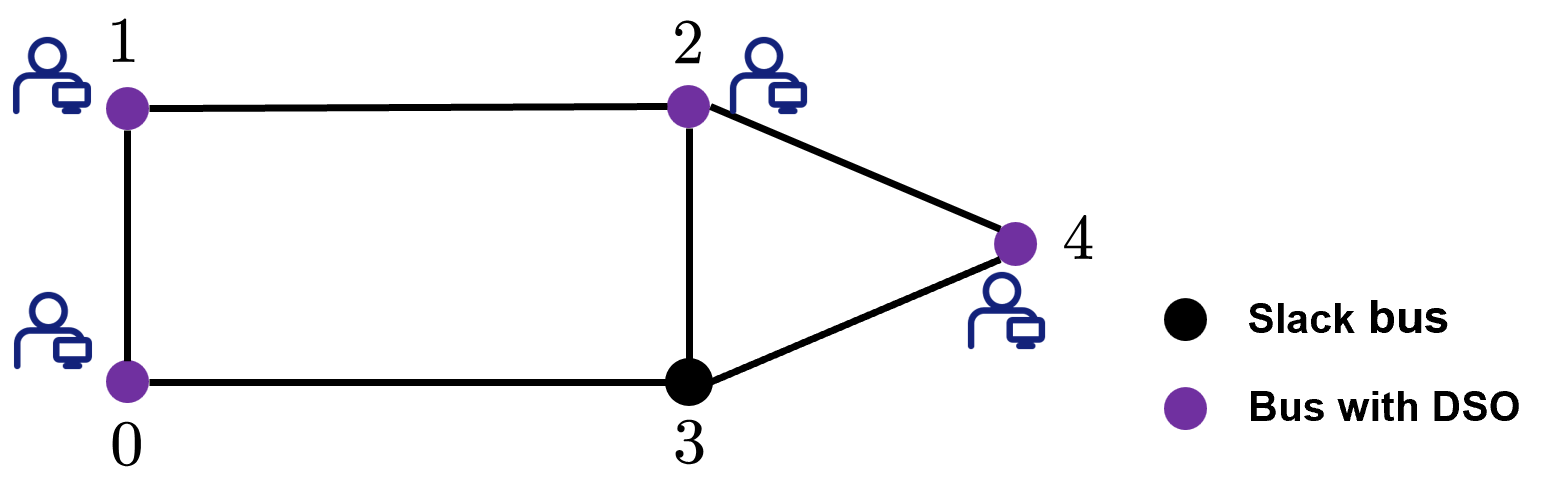}
    \caption{Modified 5-bus test system, showing TSO-DSO connections for voltage regulation studies.}
    \label{fig:5bus}
\end{figure}

The voltage constraints for secure grid operation are set as $\underline{v} = 0.96$ p.u. and $\overline{v} = 1.04$ p.u. 
The DSOs are modeled with quadratic cost functions and individual reactive power limits.
The quadratic cost coefficient is randomly selected from the range $\left[ 0.2,0.8 \right]$ to reflect varying DSO operational costs. 
The reactive power limits are initially set sufficiently large to ensure that DSOs can provide enough reactive power for effective voltage regulation, but are unknown to the TSO.

We implemented the BIG Hype-based algorithm presented in Section \ref{sec:solution} with the proposed online feedback enhancements. 
The step sizes $\eta$ and $\epsilon$ were empirically tuned to $1\times 10^{-3}$ amd $1\times 10^{-4}$, respectively, with the goal of balancing convergence speed and stability.

\subsection{Simulation results}
Our simulation results are illustrated in Figure~\ref{fig:v&v_ref}, which depicts the evolution of both voltage magnitudes and voltage references (incentive signals) at the four DSOs.

\begin{figure}[tbp] 
    \centering 
    \includegraphics[width=0.75\linewidth]{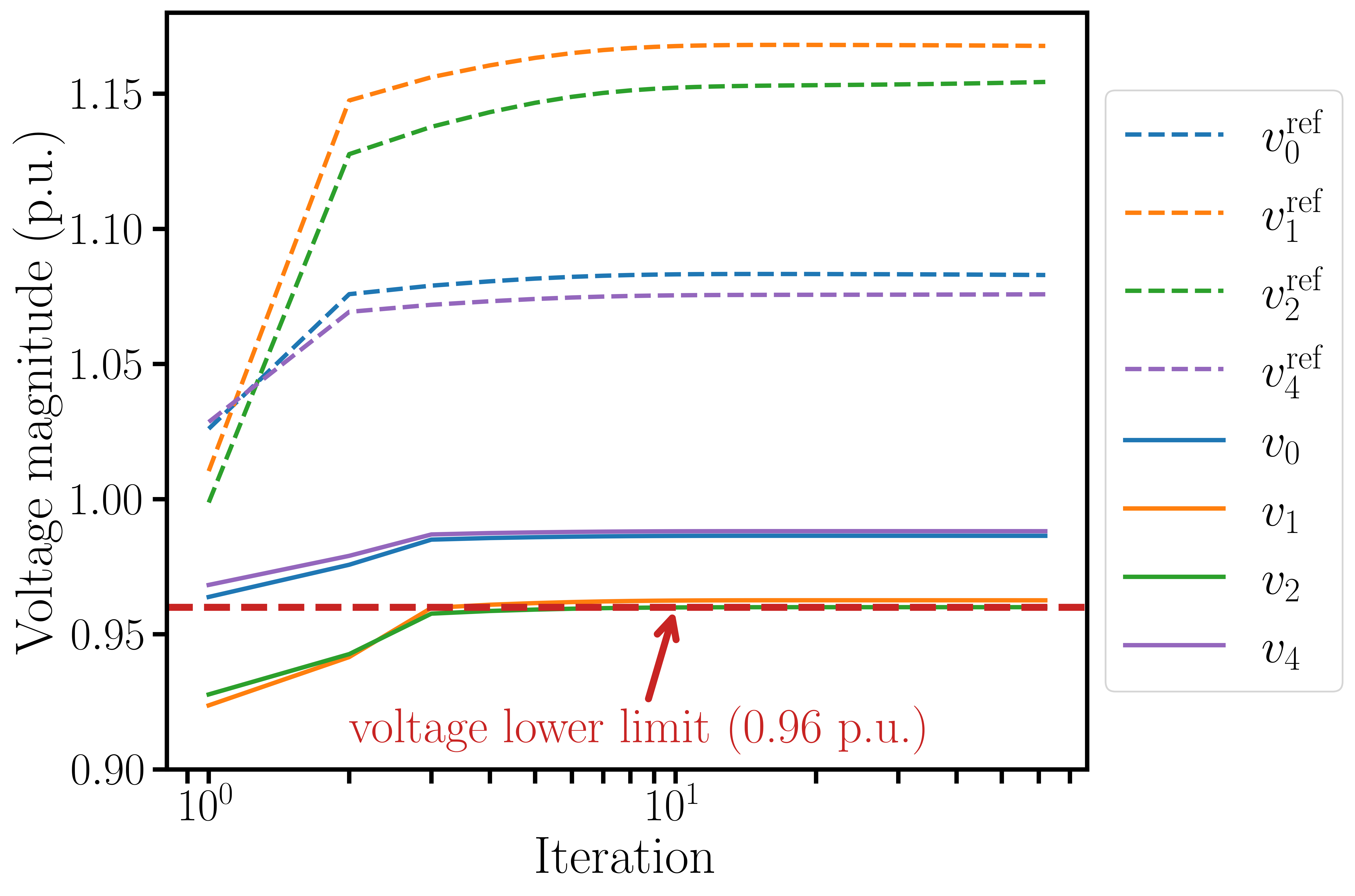} 
    \caption{Evolution of voltage magnitudes and voltage references ($v_{\text{ref}}$) at the four buses with DSO connections. The voltage reference serves as an incentive signal rather than a strict tracking target, guiding DSOs to regulate voltages within acceptable limits.} \label{fig:v&v_ref} 
\end{figure}

\begin{figure}[tbp]
    \centering
    \includegraphics[width=0.75\linewidth]{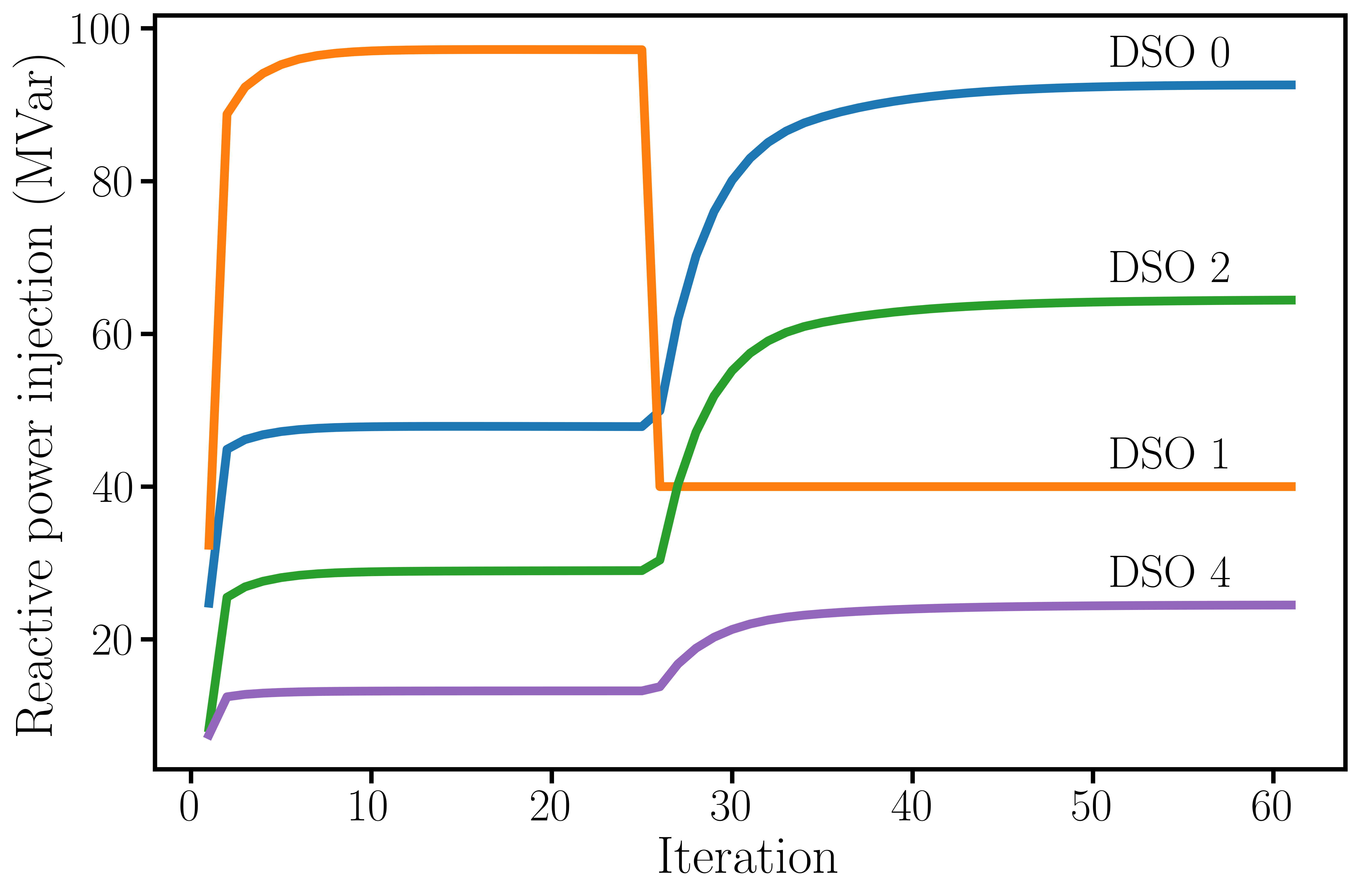}
    \caption{Reactive power responses of DSOs after the disturbance. DSO 1's injection decreases to its new limit, while other DSOs compensate to restore voltage stability.}
    \label{fig:reactive_power_after_disturbance}
\end{figure}

\begin{figure}[tbp]
    \centering
    \includegraphics[width=0.75\linewidth]{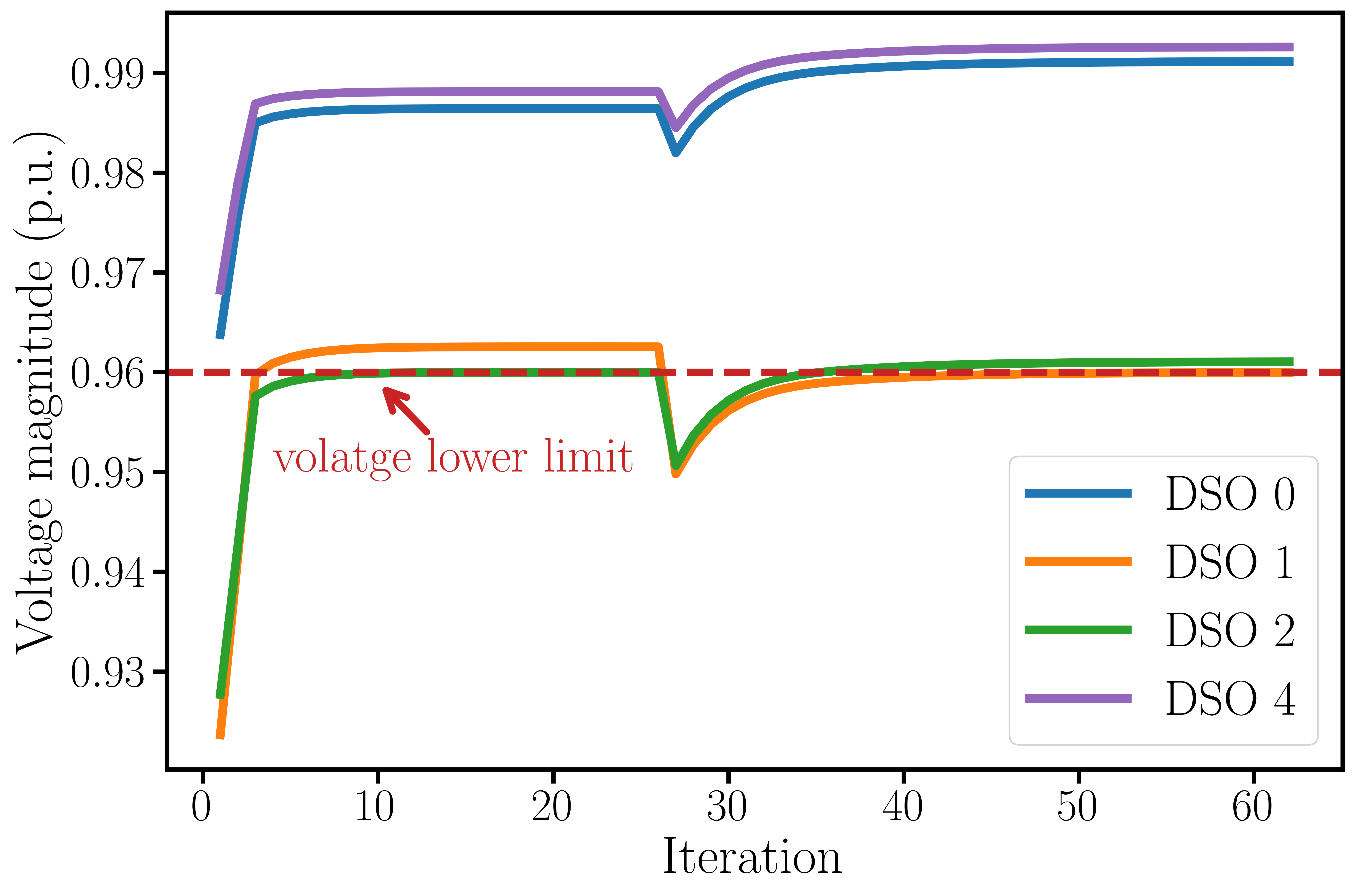}
    \caption{Voltage response after the disturbance. The voltage deviation is quickly corrected as the system stabilizes at a new equilibrium.}
    \label{fig:voltage_after_disturbance}
\end{figure}

\begin{figure}[t]
    \centering
    \includegraphics[width=0.75\linewidth]{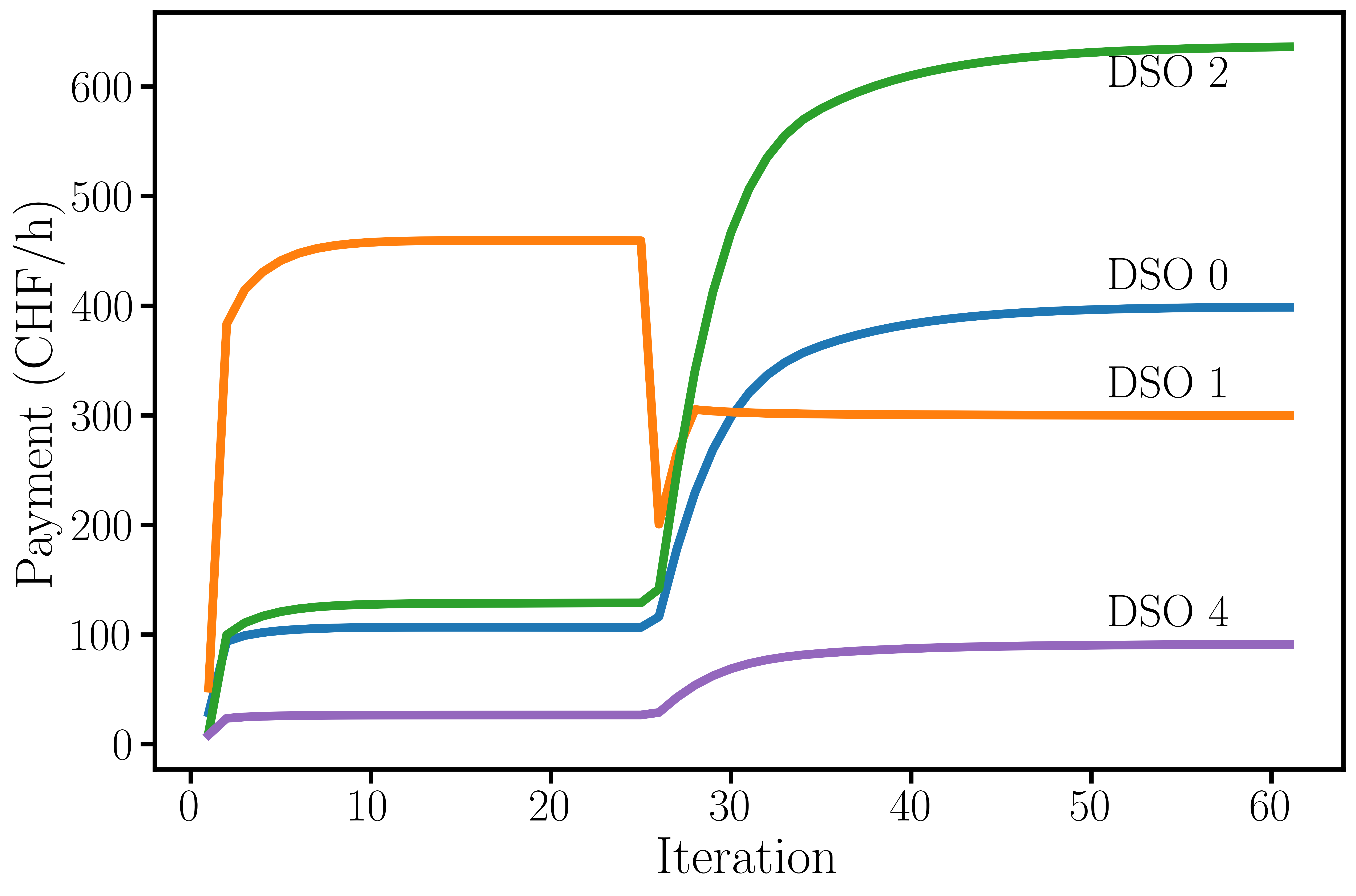}
    \caption{Adjustment of incentive payments after the disturbance. DSO 1's payment decreases due to its reduced injection capability, while other DSOs receive increased payments for their additional contributions to voltage regulation.}
    \label{fig:payment_after_disturbance}
\end{figure}

Initially, two buses' voltage magnitudes are below the lower bound. 
As the incentive updates, the voltage references dynamically adjust, steering all bus voltages into the secure operating region between 0.96 and 1.04 p.u. 
After sufficient iterations, both voltages and incentive signals stabilize, indicating convergence. Notably, we observe that the voltage reference signals are set significantly higher than the actual voltages. This is expected, as the voltage reference is not meant to be directly tracked but instead serves as a parameter to shape the incentives and, in turn, to shape the DSOs' reactive power responses, ultimately ensuring voltage regulation within limits.

To further demonstrate the advantage of our online implementation, we simulate a scenario where the reactive power injection limit of a specific DSO changes during the incentive updating process, reflecting practical situations where distribution networks' energy resources may become unavailable. 
As shown in Figure~\ref{fig:reactive_power_after_disturbance},
initially, the DSOs have reached equilibrium, and the voltage remains stable within its prescribed limits. 
At a certain point, the maximum reactive power injection capacity of DSO 1 is suddenly reduced to 40 MVar.
The reactive power responses of the DSOs adjust immediately—DSO 1's injection drops to its new limit, while other DSOs collectively increase their injections to compensate for the voltage drop. 
Figure~\ref{fig:voltage_after_disturbance} further illustrates that the voltage deviation is quickly corrected, and both the voltage and reactive power injections transition to a new stable state.

From an economic perspective, Figure~\ref{fig:payment_after_disturbance} shows how the payments adjust in response to this change. 
Due to the reduced reactive power injection limit of DSO 1, its payment from the TSO decreases accordingly. 
In contrast, the payments to other DSOs increase, reflecting their greater contribution to voltage regulation. 
This result demonstrates that the economic benefits associated with voltage support are dynamically redistributed among DSOs.

\section{Conclusions}
\label{sec:conclusions}
This paper presents a bilevel game-theoretic framework for voltage regulation, where a TSO designs voltage references as incentive signals to coordinate multiple DSOs, instead of requiring direct control of distribution-level resources. 
By leveraging a Stackelberg game formulation, we analyze the equilibrium of the lower-level noncooperative game of DSOs, propose an iterative incentive update scheme based on the BIG Hype algorithm, and conduct an online implementation utilizing online feedback optimization.
The incentive updating scheme effectively handles the voltage regulation problem's hierarchical decision structure, and we prove theoretical convergence based on the linearized grid model.
The real-time adaptation of incentives based on measured voltages demonstrates the robustness of our online implementation against changing operating conditions. 

Future work will explore more sophisticated incentive structures beyond the linear voltage error used in this paper. 
Different incentive functions have the potential to enhance the responsiveness of the response, shape the payments to the DSO according to desired characteristics, and ensure additional robustness to the DSO responses.
However, the theoretical convergence of such complex incentive functions remains an open question and warrants further investigation.

\section*{ACKNOWLEDGMENT}

The authors would like to thank Jovicic Aleksandar at Swissgrid and Leo Landolt for their help in the analysis of the Swiss voltage support incentives.

\bibliographystyle{IEEEbib}
\bibliography{refs}

\end{document}